\documentclass[11pt]{amsart}
\usepackage{graphicx,color,latexsym}
\usepackage{amssymb}
\usepackage{pb-diagram}
\usepackage{verbatim}

\usepackage{tikz}
\usetikzlibrary{decorations.pathreplacing}
\usetikzlibrary{arrows}
\usepackage{float}

\newtheorem{theorem}{Theorem}[section]

\newtheorem{lemma}[theorem]{Lemma}
\newtheorem{cor}[theorem]{Corollary}

\theoremstyle{definition}
\newtheorem{remark}[theorem]{Remark}
\newtheorem{question}[theorem]{Question}
\newtheorem{definition}[theorem]{Definition}

\newtheorem{example}[theorem]{Example}

\begin{document}

\title{Minimality of Symplectic Fiber Sums along Spheres}

\author{Josef G. Dorfmeister}
\address{Institut f\"ur Differentialgeometrie\\  Leibniz Universit\"at Hannover\\ 30167 Hannover, Germany}
\email{j.dorfmeister@math.uni-hannover.de}

\date{\today}

\begin{abstract}
In this note we complete the discussion begun in \cite{S} concerning the minimality of symplectic fiber sums.  We find that for fiber sums along spheres the minimality of the sum is determined by the cases discussed in \cite{U} and one additional case: If $X\#_VY=Z\#_{V_{\mathbb CP^2}}\mathbb CP^2$ with $V_{\mathbb CP^2}$ an embedded $+4$-sphere in class $[V_{\mathbb CP^2}]=2[H]\in H_2(\mathbb CP^2,\mathbb Z)$ and $Z$ has at least 2 disjoint exceptional spheres $E_i$ each meeting the submanifold $V_Z\subset Z$ positively and transversely in a single point with $[E_i]\cdot [V_X]= 1$, then the fiber sum is not minimal.
\end{abstract}
\maketitle

\tableofcontents

\section{Introduction}

The symplectic fiber sum has been used to great effect since its discovery to construct new symplectic manifolds with certain properties.  In four dimensions this has focused on manifolds with symplectic Kodaira dimension 1 or 2.  Moreover, it was shown in \cite{liu} and \cite{L1} that symplectic manifolds with symplectic Kodaira dimension $-\infty$ are rational or ruled;  these are rather well understood.  Hence it is reasonable to ask whether this surgery can produce new manifolds of Kodaira dimension 0.  As Kodaira dimension is defined on the minimal model of a symplectic manifold, it is first necessary to answer the question under what circumstances the symplectic sum produces a minimal manifold.

This question has been researched for fiber sums along submanifolds of genus strictly greater than 0 by M. Usher, see \cite{U}.  In this note we complete the discussion for fiber sums along spheres.  

The symplectic fiber sum is a surgery on two symplectic manifolds $X$ and $Y$, each containing a copy of a symplectic hypersurface $V$.  The sum $X\#_VY$ is again a symplectic manifold.  Section \ref{pre} provides a brief overview of the symplectic sum construction and minimality of symplectic manifolds.  A minimal symplectic manifold contains  no exceptional spheres, i.e. no embedded symplectic spheres of self-intersection $-1$.  Furthermore, we review the main tool for proving the following theorem, namely the symplectic sum formula for Gromov-Witten invariants.  Applying this formula in the genus 0 case will involve a detailed look at the behavior of relative curves in relation to the hypersurface $V$.  

Symplectic sums along spheres can be classified due to the following result by McDuff:  Thm. 1.4 in \cite{M} implies that one of the summands must be contained in the following list:
\begin{itemize}
\item $(Y,V_Y)=(\mathbb CP^2,H)$, $[H]$ the generator of $H_2(\mathbb CP^2,\mathbb Z)$,
\item $(Y,V_Y)=(\mathbb CP^2,2H)$,
\item $Y$ a $S^2$-bundle over a genus $g$ surface, $V_Y$ a fiber, or
\item $Y$ a $S^2$-bundle over $S^2$, $V_Y$ a section.
\end{itemize} 
Section \ref{spheresum} provides a number of examples of manifolds which can be involved in a symplectic sum along a sphere.  Emphasis is placed on the case of a $-4$-sphere as this provides the most intriguing examples.   

A case by case study in Section \ref{min} considers the main topic of this note, the minimality of fiber sums $M=X\#_VY$ along symplectic spheres $V$ in $X$ and $Y$.  The minimality of sums containing an $S^2$-bundle as one of the summands is identical to the case of higher genus sums considered in \cite{U}.  The arguments concerning the $\mathbb CP^2$ summands make use of the sum formula for GW-invariants, and in particular the case of $(\mathbb CP^2,2H)$, the rational blow-down of a $-4$-sphere, is interesting.  In contrast to the positive genus case of \cite{U}, exceptional spheres produced in the sum along a sphere may contain contributions from curves with components lying in the the hypersurface $V$.  The arguments involving the sum formula must take into account this subtle issue.

For rational blow-downs, we show that the sum is not minimal in the presence of a certain configuration of curves.  This leads to the following criterion for minimality:

\begin{lemma}Let $M=X\#_{V_X=2H}\mathbb CP^2$ be the rational blow-down of an embedded symplectic $-4$-sphere $V_X$ in $X$.  Then $M$ is not minimal if:
\begin{itemize}
\item $X\backslash V_X$ is not minimal or 
\item $X$ contains  2 disjoint distinct exceptional spheres $E_i$ each meeting $V_X$ transversely and positively in a single point with $[E_i]\cdot [V_X]=1$.
\end{itemize}

\end{lemma}

This result, together with the genus $g>0$ case proven in \cite{U}, completely answers the question of the minimality of the symplectic fiber sum in dimension 4:

\begin{theorem}\label{minimal}
Let $M$ be the symplectic fiber sum $X\#_VY$ of the symplectic manifolds $(X,\omega_X)$ and $(Y,\omega_Y)$ along an embedded symplectic surface $V$ of genus $g\ge 0$.  
\begin{enumerate}
\item The manifold $M$ is not minimal if
\begin{itemize}
\item $X\backslash V_X$ or $Y\backslash V_Y$ contains an embedded symplectic sphere of self-intersection $-1$ or
\item $X\#_VY=Z\#_{V_{\mathbb CP^2}}\mathbb CP^2$ with $V_{\mathbb CP^2}$ an embedded $+4$-sphere in class $[V_{\mathbb CP^2}]=2[H]\in H_2(\mathbb CP^2,\mathbb Z)$ and $Z$ has at least 2 disjoint exceptional spheres $E_i$ each meeting the submanifold $V_Z\subset Z$ positively and transversely in a single point with $[E_i]\cdot [V_X]= 1$.
\end{itemize}
\item If $X\#_VY=Z\#_{V_B}B$ where $B$ is a $S^2$-bundle over a genus $g$ surface and $V_B$ is a section of this bundle then $M$ is minimal if and only if $Z$ is minimal.  
\item In all other cases $M$ is minimal.
\end{enumerate}
\end{theorem}

The non-minimal setup in the rational blow-down is the following gluing:

\begin{figure}[H]
\centering
\begin{tikzpicture}
\draw (-2,4) node{$X$};
\draw (3,4) node{$\mathbb{C}P^2$};
\draw (0,0)--(0,4) node[anchor=east]{$V_X$};
\draw (0,1)--(-2,1) node[anchor=north]{$E_2$};
\draw (0,3)--(-2,3) node[anchor=north]{$E_1$};
\draw (1,0)--(1,4) node[anchor=west]{$V_{\mathbb CP^2}$};
\draw[dotted] (0,3)--(1,3);
\draw[dotted] (0,1)--(1,1);
\draw (1,1) arc(-90:90:1cm);
\draw (2,2) node[anchor=south west]{$H$};
\end{tikzpicture}
\caption{The criterion case for non-minimality in the rational blow-down of a symplectic $-4$-sphere $V_X$.  }
\end{figure}
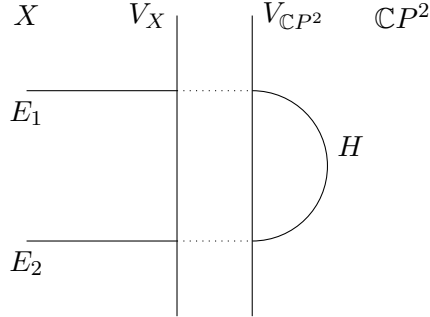


{\bf Acknowledgments}  I gratefully acknowledge the patient support of my advisor Prof. Tian Jun Li.    The exposition of this paper as well as its mathematical content were considerably improved by the precise reading of Prof. Dusa McDuff.  I also thank Prof. Bob Gompf, Prof. Anar Akhmedov,  Weiyi Zhang and Weiwei Wu for their interest.  Part of this work was completed while at the University of Minnesota.  The author is supported through the Graduiertenkolleg 1463: Analysis, Geometry and String Theory.  

\section{Preliminaries\label{pre}}

In this section we provide an overview of the methods used in the proofs.  This is to make this note self-contained and to introduce notation used in the rest of the paper.

\subsection{Symplectic fiber sum}

The symplectic fiber sum is a smooth surgery operation which is performed in the symplectic category.    Let $X_1$, $X_2$ be $2n$-dimensional smooth
manifolds. Suppose we are given
codimension 2 embeddings $j_i:V\rightarrow X_i$ of  a
smooth closed oriented manifold $V$ with normal bundles $N_iV$.
Assume that the Euler classes of the normal bundle of the embedding
of $V$ in $X_i$ satisfy $e(N_1V)+e(N_2V)=0$ and fix a
fiber-orientation reversing bundle isomorphism $\Theta:
N_1V\rightarrow N_2V$.  By canonically identifying the normal
bundles with a tubular neighborhood $\nu_i$ of $j_i(V)$, we obtain
an orientation preserving diffeomorphism $\varphi: \nu_1\backslash
j_1(V)\rightarrow \nu_2\backslash j_2(V)$ by composing $\Theta$ with
the diffeomorphism that turns each punctured fiber inside out.  This
defines a gluing of $X_1$ to $X_2$ along the embeddings of $V$ denoted
$M=X_1\#_{(V,\varphi)}X_2$.  The diffeomorphism type of this manifold is
determined by the embeddings $(j_1,j_2)$ and the map $\Theta$.  

That this procedure works in the symplectic category is due to Gompf (\cite{G}) and McCarthy-Wolfson (\cite{MW}).  Assume $X_1$ and $X_2$ admit symplectic forms $\omega_1,\omega_2$ resp.  If the embeddings $j_i$ are symplectic with respect to these forms, then
we obtain $M=X_1\#_{(V,\varphi)}X_2$ together with a symplectic form
$\omega$ created from $\omega_1$ and $\omega_2$.

This procedure can be reversed, this is called the symplectic cut, see \cite{Le} for details.  Throughout this paper, we suppress the map $\varphi$ in the notation and work in the symplectic category.  Hence we refer only to the fiber sum $M=X_1\#_VX_2$.

\subsection{Notation}

Let $(X,\omega_X)$ be a smooth, closed, symplectic 4-manifold.  We shall generally not distinguish between a symplectic form $\omega$ and a symplectic class $[\omega]$, i.e. a cohomology class which can be represented by a symplectic form.  Both shall be denoted by $\omega$.  Let $V$ be a $2$-dimensional smooth closed manifold such that $X$ contains a copy $V_X$ and $V_Y$ of $V$ which is symplectic with respect to $\omega_X$.  Such a triple $(X,V_X,\omega_X)$ will be called a symplectic pair and we will suppress the $\omega_X$ in the notation henceforth.  Denote the homology class of $V_X$ by $[V_X]\in H_2(X)$ and the fiber sum of two sympelctic pairs $(X,V_X)$ and $(Y,V_Y)$ along $V\cong V_X\cong V_Y$ by  $M=X\#_VY$.

Let $A\in H_2(M,\mathbb Z)$ be represented by a connected surface $C\subset M$ of genus $g$.  The surface $C$ decomposes under a symplectic cut into surfaces $C_X\subset X$ and $C_Y\subset Y$, each surface not necessarily connected.  We shall denote the class of $C_X$ by $A_X$ and similarly for $C_Y$.  Let $r$ denote the number of intersection points of the surface $C_X$ with $V_X$, counted without multiplicities.  This is of course the same as for $C_Y$ and $V_Y$.  On the level of homology, we have  $A_X\cdot [V_X]\ge r\ge 0$.

A class $K_X\in H^2(X,\mathbb Z)$ is called a symplectic canonical class if there exists a symplectic form $\omega$ on $M$ such that for any almost complex structure $J$ tamed by $\omega$,
\[
K_X=K_\omega=-c_1(X,J).
\]
We will suppress the dependence on the symplectic form $\omega$ as our calculations will be unaffected by the precise choice of $\omega$.

\subsection{Splitting of classes under fiber sums}

As described in the introduction, we will apply the sum formula for Gromov-Witten invariants to determine the minimality of the sum $M=X\#_VY$.  For this reason we need to understand how homology and cohomology classes behave under the symplectic fiber sum.  To begin, we have the following:

\begin{lemma}\label{formulas}Given a fiber sum $M=X\#_VY$ and using the above notation, we have
\begin{equation}
K_M\cdot A=(K_X+[V_X])\cdot A_X+(K_Y+ [V_Y])\cdot A_Y,
\end{equation}
and
\begin{equation}
A^2=A_X^2+A_Y^2
\end{equation}
where $A^2=A\cdot A$.

\end{lemma}

\begin{proof}
The first is  Lemma 2.3, \cite{IP5}.  The second is in \cite{Liun}, we provide a sketch here.  Consider the map $\pi:M\rightarrow X\cup_{V_X=V_Y}Y$ and a decomposition $\pi_*A=A_X+A_Y$.  Representing $A_X$ and $A_Y$ by appropriate closed surfaces $C_X$ and $C_Y$ resp., we obtain a closed surface $C=\pi^{-1}(C_X\cup C_Y)$ in $M$.  The class of this surface must be $A+e$ for some $e\in\ker(\pi_*)$.  Judiciously choosing a second set of representatives $C_X'$ and $C_Y'$ and repeating this procedure to obtain a surface $C'$, we can calculate the intersection number of $A+e$ and $A+e'$ from intersections of the surfaces $C_*$ and $C_*'$.  This leads to the formula.
\end{proof}

\subsection{Minimality of 4-Manifolds}

\begin{definition} Let ${\mathcal E}_X$ be the set of cohomology
classes whose Poincar\'e dual are represented by smoothly embedded
spheres of self-intersection $-1$. $X$ is said to be (smoothly)
minimal if ${\mathcal E}_X$ is the empty set.
 \end{definition}

Equivalently, $X$ is minimal if it is not the connected sum of
another manifold $Y$  with $\overline{\mathbb {CP}^2}$. 

\begin{definition}
The manifold $X_m$
is called a minimal model of $X$ if $X_m$ is  minimal and $X$ is the
connected sum $X_m\#k\;\overline{\mathbb {CP}^2}$ for some $k>0$.
\end{definition}

The notion of minimality is also defined for a symplectic manifold $(X,\omega)$:
$(X,\omega)$  is said to be (symplectically) minimal if ${\mathcal
E}_{\omega}$ is the empty set, where
$${\mathcal E}_{\omega}=\{E\in {\mathcal
E}_X|\hbox{ $E$ is represented by an embedded $\omega-$symplectic
sphere}\}.$$ 

Such a sphere is called an exceptional sphere, its class an exceptional class.  A basic fact proven using SW theory (\cite{T1},
\cite{LL}, \cite{TJL2}) is:

\begin{lemma}
  ${\mathcal E}_{\omega}$ is
empty if and only if ${\mathcal E}_X$ is empty. In other words, $(X,
\omega)$ is symplectically minimal if and only if $X$ is smoothly
minimal.

\end{lemma}

A manifold $X$ is called rational if its underlying smooth manifold is either $(S^2\times S^2)\#k\;\overline{\mathbb CP^2}$ or $\mathbb CP^2\#k\;\overline{\mathbb CP^2}$ for some $k\ge 0$.  A manifold $X$ is ruled if the underlying smooth manifold is a connected sum of a $S^2$-bundle over a Riemann surface with $k$ copies of $\overline{\mathbb CP^2}$, $k\ge 0$.  The following two results will be useful:

\begin{lemma}\label{inters}\begin{enumerate}
\item (Thm. 1.5, \cite{M2}) If $X$ is not rational or ruled, then $X_m$ is unique.  

\item (Cor. 3, \cite{TJL2})
If $X$ is not rational or ruled, then no two distinct exceptional spheres intersect.
\end{enumerate}
\end{lemma}

The process for obtaining a minimal manifold from $X$ is called blowing down.  This removes the $-1$-sphere and can be obtained as a fiber sum of $(X,E)$ with $(\mathbb CP^2,H)$, i.e. the blown-down manifold $Y=X\#_{E=H}\mathbb CP^2$.  It can be shown, that after blowing down a finite collection of exceptional spheres, a minimal manifold is obtained, see Thm 1.1, \cite{M}.

The symplectic fiber sum involves a submanifold $V_X\subset X$.  It is therefore reasonable to consider minimality with respect to this submanifold:

\begin{definition} 
The pair $(X,V_X)$ is called relatively minimal if there exist no exceptional spheres $E$ such that $E\cdot [V]=0$.
\end{definition}

The following result shows that we can blow down $X$ in such a manner that $V$ is preserved and the result is relatively minimal:

\begin{lemma}[Thm 1.1ii, \cite{M}] Every symplectic pair $(X,V_X)$ covers a relatively minimal symplectic pair $(\tilde X,V_X)$ which may be obtained by blowing down a finite set of exceptional spheres disjoint from $V$.

\end{lemma}

\subsection{Gromov-Witten Invariants and Fiber Sums}

The main tool in the proof of minimality will be the sum formula for Gromov-Witten invariants.  This formula relates the absolute GW-invariants of $M$ to the relative GW-invariants of $(X,V_X)$ and $(Y,V_Y)$.  This section introduces relevant notation and briefly describes the setup.

\subsubsection{Connected Invariants}

A stable pseudoholomorphic map is a pair $(\Sigma, \phi)$ consisting of a connected Riemann surface $\Sigma$ of genus $g$ and a pseudoholomorphic map $\phi:\Sigma\rightarrow X$ such that the pair $(\Sigma,\phi)$ has finite automorphism group.  A curve is an equivalence class of such maps, two maps $(\Sigma,\phi)$ and $(\Sigma',\phi')$ being equivalent if the diagram
\[
\begin{diagram}
\node{\Sigma}\arrow{s}\arrow{e,t}{\phi}\node{X}\\
\node{\Sigma'}\arrow{ne,r}{\phi'}
\end{diagram}
\]
commutes.  If we have additional data on $\Sigma$, for example markings, then the diagram is expected to respect this data as well.  Equivalently, we call the image $\phi(\Sigma)$ a curve in $X$.

In short, absolute Gromov-Witten invariants $\langle \alpha_1,...,\alpha_n\rangle^X_{g,A}$ are numerical invariants of the symplectic manifold $(X,\omega)$ associated to equivalence classes of stable pseudoholomorphic maps $(\Sigma,\phi)$ representing the class $A\in H_2(X)$ with genus $g$ and meeting representatives of the Poincar\'e dual classes of $\alpha_i\in H^*(X)$.  One method to calculate this number is to integrate an appropriate choice of classes over the virtual fundamental cycle of the appropriate moduli space of curves, details can be found in \cite{LR}.  The dimension of this virtual cycle equals the index associated to these curves, and hence, for the GW-invariant to be a meaningful number, the sum of the dimensions of the constraints must equal the index:
\[
\sum\mbox{ deg }\alpha_i\;=\; 2\left(-K_X\cdot A+(g-1)+n \right).
\]  
If this is not the case, the GW-invariant is defined to be zero.

\begin{example}\label{-1}
Let $E\subset X$ be an exceptional sphere, i.e. the embedded image of a sphere with $E^2=-1$.  Moreover, by the adjunction formula, $K_X\cdot E=-1$, and by positivity of intersections, we can have at most one such curve for a given $J$.  Furthermore, the dimension of the moduli space of exceptional spheres is
\[
-K_X\cdot A+(g-1)=0
\]
and thus we have no insertions in the Gromov-Witten invariant save an $X$ insertion (which we omit), which provides no constraint.  In particular, we obtain (see for example \cite{MS})
\[
\langle\; \rangle^X_{0,E}=1.
\]
\end{example} 

The relative Gromov-Witten invariant $\langle \alpha_1,...,\alpha_n\vert\beta_1,...,\beta_r\rangle^{X,V_X}_{g,A,{\bf s}}$ has a similar interpretation, however the underlying moduli space of stable pseudoholomorphic maps contains only such curves which, in addition to meeting representatives of the Poincar\'e dual classes of $\alpha_i\in H^*(X)$,  meet the hypersurface $V_X\subset X$ at points $\{p_1,..,p_r\}$ with prescribed orders ${\bf s}=(s_1,..,s_r)$ and at representatives $B_i$ of the Poincar\'e duals of the classes $\beta_i\in H^*(V)$ with $p_i\in B_i$.

As before, we must have
\[
\sum\mbox{ deg }\alpha_i\;+\;\sum \mbox{ deg }\beta_i\;=\; 2\left(-K_X\cdot A+(g-1)+n+r-\sum s_i \right)
\]
with $\sum s_i=A\cdot[V_X]$ for the invariant to perhaps be non-trivial.  Note that the $\beta_i$ lie in the cohomology of $V$, hence can have at most degree 2, while the classes $\alpha_i$ lie in the cohomology of $X$ and thus can have degree at most 4.

It should be noted, that we will only be interested in the case that $V_X$ is a sphere.  Thus $H_1(V_X)=0$ and we may avoid any rim tori discussions, see \cite{IP4}.

\subsubsection{Disconnected Invariants}

Our main interest will be focused on connected curves in the sum $M=X\#_VY$ which we will research with regard to the components in $X$ and $Y$ obtained after cutting $M$.  These component curves $C_X$ and $C_Y$ will be regarded relative to a symplectic hypersurface $V$, but they need not be connected.  We now introduce notation for disconnected relative invariants.

Let $T$ be a labelled graph with tails.  Each tail will be labelled with an $s_i$ and denote the contact order of intersection with $V$ of the node from which the tail emanates.  The graph can be disconnected and, if the genus of the curve is 0, the graph will be a labelled tree.  We define a stable pseudoholomorphic map $(T,\phi)$ on this graph and, by a commutative diagram as in the connected case, an equivalence class of such maps.  As before, we call such classes curves. 

In accordance with the relative notation described above, denote a relative invariant by $\langle \alpha_1,...,\alpha_n\vert\beta_1,...,\beta_r\rangle^{X,V_X}_{T,{\bf s}}$, where $T$ contains the necessary information on the arithmetic genus and class of the curve.  The ordering of the $\beta_i$ and $\bf s$ will associate a cohomology class to each tail of $T$.

The disconnected invariant is defined as the product of the connected relative invariants $\langle \alpha_1,...,\alpha_{n_k}\vert\beta_1,...,\beta_{r_k}\rangle^{X,V_X}_{g_k,A_k,{\bf s}_k}$ of each connected component $T_k$ of $T$.

\subsubsection{The Sum Formula for GW-Invariants}
The sum formula developed in \cite{LR} relates these two invariants:  If $M=X\#_VY$ is a symplectic fiber sum, then 
\begin{equation}
\langle\alpha_1,...,\alpha_n\rangle^M_{g,A}=
\end{equation}
\[
=\sum \langle\;\tilde\alpha_1,...,\tilde\alpha_{n_1}\vert\beta_1,...,\beta_k\rangle^{X,V_X}_{T_X,{\bf s}}\;\delta_{T,{\bf s}}\;\langle\tilde\alpha_{n_1+1},...,\tilde\alpha_{n_2}\vert \hat\beta_1,...,\hat\beta_k\rangle^{Y,V_Y}_{T_Y,{\bf s}}
\]
where the sum is taken over all possible connected labelled graphs $T=T_X\cup T_Y$, decompositions of the $\alpha_i$ and possible insertions $\beta_i$.  The classes $\hat\beta$ are dual to the classes $\beta$ in $H^*(V)$.  The number $\delta_{T,{\bf s}}$ takes into account automorphisms of the labelled graph $T$, see \cite{Li1}.  For the purposes of determining the number of exceptional spheres in the symplectic sum, this number will be 1 throughout.

Intuitively, this formula states that any $J$-holomorphic curve in $M$ must be produced in the sum from the combination of two curves in $X$ and $Y$, keeping in mind that such a curve could lie wholly in $X$ or $Y$, and that therefore the absolute invariant on $M$ should be calculable by combining all possible allowed curves in $X$ and $Y$.  It should be emphasized again, that we must allow the curves in the summands to be disjoint tuples of curves such that when summed we obtain a connected genus $g$ curve representing the class $A$.  For example, we could have two disjoint curves of genus $0$ in $X$ representing the classes $A_1$ and $A_2$ such that $A_1^2+A_2^2=-2$ summed with a curve in class $[H]$ in $\mathbb CP^2$ to obtain a connected genus $0$ curve with square $-1$.  Configurations of this type must be carefully considered when calculating the dimension of the respective moduli spaces of maps.

\subsubsection{Degenerate Curves}
The intuitive understanding of the sum formula is easily seen to be reasonable, as long as one can prevent curves in $X$ or $Y$ from having components that lie in the hypersurface $V$.  Curves with components in the hypersurface $V$ would appear to lose these parts under the fiber sum operation, yet the whole curve would contribute to the relative invariants on the right side of the sum formula.  Moreover, in the moduli space of curves, any sequence of curves will converge to a limit by Gromov compactness and we cannot a priori assume that this limit does not have components mapping into $V$.  This is a very subtle issue, we now describe how to handle such limit curves.  This is described in detail in \cite{M3} and \cite{HLR} (see also \cite{LR} and \cite{Li1}, \cite{M3} contains numerous examples of this construction).  

The idea is to extend the manifold $X$ in such a manner, that the components descending into $V_X$ get stretched out and become discernible.  This extension is achieved by gluing $X$ along $V_X$ to the projective completion of the normal bundle $N_{V_X}$.  This completion is denoted by $Q=\mathbb P(N_{V_X}\oplus \mathbb C)$ and it comes with a natural fiberwise $\mathbb C^*$ action.  The ruled surface $Q$ contains two sections, the zero section $V_0$, which has opposite orientation to $V_X$, and the infinity section $V_\infty$, which is a copy of $V_X$ with the same orientation both of which are preserved by the $\mathbb C^*$ action.  The manifold $X\#_{V_X=V_0}Q$ is symplectomorphic to $X$ and can be viewed as a stretching of the neighborhood of $V_X$.  This stretching can be done any finite number of times.  Therefore consider the singular manifold $X_m=X\sqcup_{V=V_0}Q_1\sqcup_{V_\infty=V_0}...\sqcup_{V_\infty=V_0}Q_m$ which as been stretched $m$ times.  This will provide the extended target for the preglued curves which we now describe.  

Any map $(\Sigma, \phi)$ into $X$ with components lying in $V_X$ can be viewed as a map into $X_m$ consisting of a number components:  Decompose $\Sigma$ into $\Sigma_i$, possibly disconnected, such that $\phi(\Sigma_i)\subset Q_i$ (with $X=Q_0$).  Then each curve $C=\phi(\Sigma)$ defines a decomposition into levels $C_i$ which lie in $Q_i$.  Each must satisfy a number of contact conditions:  $C_i$ and $C_{i+1}$ contact along $V$ in their respective components $(Q_i,V_\infty)$ and $(Q_{i+1},V_0)$ such that contact orders and contact points match up.  The imposed contact conditions on $V_X$ given in $\bf s$ are imposed on the level $C_m$ where it contacts $V_\infty$ of $Q_m$.  

The curves $C_i$, viewed as maps into $X_m$, must satisfy certain stability conditions.  In $X$, these are the well known standard conditions on the finiteness of the automorphism group.  For those mapping into $Q_i$, $i>0$, we identify any two submanifolds which can be mapped onto each other by the $\mathbb C^*$ action of $Q$.

We obtain a curve $C_g$ in $X$ meeting $V_X$ as prescribed by $\bf s$ from $C_0\cup C_1\cup...\cup C_m$ by gluing along $V$ in each level.  The homology class of the preglued curve $C$ is calculated as the sum of the homology class of $C_0$ and the projections of the class of $C_i$ into $H_2(V)$.  The projections in $4$-dimensions are just the classes $[V]$ (or  multiples of $[V]$) and $[pt]$ in $H_2(V)$, the images of a section and a fiber respectively.

For each such curve $C$ representing the class $A$, it is possible to determine the index of the associated differential operator and hence determine the dimension of the strata in which it lies from this construction.  In the following calculation we shall assume that there are no absolute insertions, this is the case of interest in in Section \ref{min}.  At level $i>0$ denote the following:
\begin{itemize}
\item $r_{i0}$ and $r_{i\infty}$ are the number of contact points of the curve $C_i\subset Q$ with $V_0$ and $V_\infty$ respectively,
\item the (arithmetic) genus of $C_i$ is $g_i$ and
\item $[V_0]$ and $[V_\infty]$ the classes of the zero and infinity divisor in $H_2(Q)$.
\end{itemize}
Then the index of $C_i$ is given by 
\[
I_i=-K_Q\cdot [C_i]+(g_i-1)+r_{i0}-[C_i]\cdot [V_0]+r_{i\infty}-[C_i]\cdot [V_\infty]-1,
\]
where the last subtraction is due to the $\mathbb C^*$-action.  At level $i=0$ we have the index 
\[
I_0=-K_X\cdot [C_0]+(g_0-1)+r_{0}-[C_0]\cdot [V_X].
\]
Note that $[C_i]\cdot [V_\infty]=[C_{i+1}]\cdot [V_0]$ and $r_{i+1\;0}=r_{i\infty}$.  The latter implies $\sum _{i=1}^mr_{i0}=\sum_{i=0}^{m-1}r_{i\infty}$.  Moreover, a curve with $r$ contact points on $V$ has $r=r_{m\infty}$.  The index $I_C$ of the $m+1$ level curve $C$ is then given by 
\[
I_C= \sum_{i=0}^mI_i-2\sum_{i=1}^m r_{i0}
\]
where the last term accounts for the transversality conditions at the intersections of the $C_i$ which are needed to ensure that we can glue to obtain a curve in $X$ as described above.  Combining these results with Lemma \ref{formulas} and using the fact that the genus $g$ of the glued curve $C$ is given by \[
g=\sum_{i=0}^m g_i+\sum_{i=1}^m(r_{i0}-1)=\sum_{i=0}^m g_i+\sum_{i=1}^mr_{i0}-m
\]
allows us to calculate the index of a curve $C$ obtained by gluing $m+1$ levels $\{C_i\}$ with no absolute markings:

\begin{eqnarray}
\nonumber I_C&=&\sum_{i=0}^mI_i-2\sum_{i=1}^m r_{i0}\\
\nonumber &=&-K_X\cdot [C_0]+(g_0-1)+r_{0\infty}-[C_0]\cdot [V_X]\\
\nonumber &&+\sum_{i=1}^m\{-K_Q\cdot [C_i]+(g_i-1)+r_{i0}-[C_i]\cdot [V_0]+r_{i\infty}\\
\nonumber &&-[C_i]\cdot [V_\infty]-1\}-2\sum_{i=1}^m r_{i0}\\
\nonumber &=&-\left(K_X\cdot[C_0]+\sum_{i=1}^m \left\{K_Q\cdot[C_i]+2[C_i]\cdot[V_0]\right\}\right)\\
\nonumber &&+\sum_{i=0}^m (g_i-1)+\sum_{i=0}^{m-1}r_{i\infty}+r_{m\infty}-A\cdot [V_X]-m-\sum_{i=1}^m r_{i0}\\
\nonumber &=&-K_X\cdot A+ \left(\sum_{i=0}^m(g_i-1)+\sum_{i=1}^mr_{i0}\right)+r-A\cdot [V_X]-m-\sum_{i=1}^m r_{i0}\\
\nonumber &=&-K_X\cdot A+(g-1)+r-A\cdot [V_X]-m-\sum_{i=1}^m r_{i0}.
\end{eqnarray}

Thus we obtain

\begin{lemma}\label{index} Assume $(X,V)$ is a symplectic pair and dim $X=4$.  Let $C$ be a preglued curve with $m+1$ levels representing the class $A\in H_2(X)$ with no absolute marked points, $r$ relative marked points on $V$ and prescribed contact orders with $V$ given in ${\bf s}=(s_1,...,s_r)$.  Then, for generic almost complex structures $J$ among those making $V_X$ pseudoholomorphic, the curve $C$ lies in a strata of dimension no larger than
\begin{equation}
I(A,g)-m
\end{equation}
where $I(A,g)=-K_X\cdot A+(g-1)+r-\sum s_i$ and $\sum s_i=A\cdot [V_X]$.
\end{lemma}
It follows, that such degenerate curves lie in strictly lower dimensional strata of the moduli space and hence, for generic $J$, do not contribute to the GW-invariants in the formula above.  We may thus discount such curves in the Gromov-Witten theoretic arguments of the following section.

\section{Symplectic Sums along Spheres\label{spheresum}}

In four dimensions, the condition on the Euler classes of the normal bundles of $V$ in $X$ resp. $Y$ can be written in terms of the squares of the classes of $V$:
\[
[V_X]^2+[V_Y]^2=0.
\]
We shall assume that $[V_Y]^2\ge 0$.  As we are summing along a sphere, Thm. 1.4, \cite{M}, shows that we must have one of the following four cases for $(Y,V_Y)$ (This can also be found in \cite{G}.):
\begin{itemize}
\item $(Y,V_Y)=(\mathbb CP^2,H)$, 
\item $(Y,V_Y)=(\mathbb CP^2,2H)$,
\item $Y$ a $S^2$-bundle over a genus $g$ surface, $V_Y$ a fiber, or
\item $Y$ a $S^2$-bundle over sphere, $V_Y$ a section.
\end{itemize} 

The last case, with $V_Y$ a section, has already been used in Section \ref{pre} when describing higher level curves.  The manifold $Q$ is precisely such a ruled manifold.  The first case would imply that $V_X$ is an embedded sphere of self-intersection $-1$, hence this is just the standard blow-down of an exceptional sphere.

A symplectic sum with $(Y,V_Y)=(\mathbb CP^2,2H)$ is called a rational blow down of a $-4$-sphere in $X$.  This will be the most interesting case to be considered, so we provide some examples.

\begin{example}\label{easy}  It is easy to construct a symplectic $-4$-sphere in any symplectic manifold $X$:  Blow-up a point $p\in X$ to obtain an exceptional sphere $e_1$.  Then blow-up three distinct points on $e_1$ to change the first exceptional sphere into an embedded symplectic  $-4$-sphere $V$.  Note that $[e_1]\cdot [V]=-1$.

\end{example}

In particular, this type of $-4$-sphere is the only type to be found in irrationally ruled surfaces:

\begin{example}[Ruled Surfaces]  Let $X$ be an irrationally ruled surface, i.e. a surface with minimal model a $S^2$-bundle over a Riemann surface $\Sigma_g$ with $g>0$.  For a given surface $\Sigma_g$, there are only two $S^2$-bundles over it, the trivial bundle $\Sigma_g\times S^2$ and the nontrivial bundle $\Sigma_g\tilde\times S^2$.  Note that $(\Sigma_g\times S^2)\#k\;\overline{\mathbb CP^2}$ is symplectomorphic to $(\Sigma_g\tilde\times S^2)\#k \;\overline{\mathbb CP^2}$ for $k>0$.

Clearly, any rational curve must lie in a fiber of the bundle, hence in particular neither $\Sigma_g\times S^2$ nor $\Sigma_g\tilde\times S^2$ contain a $-4$-sphere.  Denoting the fiber class by $f$ and the classes of the $k$ exceptional spheres by $e_i$, a rational curve must lie in class $af-\sum_{i=1}^k a_ie_i$.  In the case of a $-4$-sphere, we obtain the relations
\[
-4=-\sum_{i=1}^k a_i^2\mbox{    and    }2=-2a+\sum_{i=1}^k a_i.
\]
From this it clearly follows that $k\le 4$ for the pair $(X,V_X)$ to be relatively minimal.  Moreover,  the only solutions to the first equation are
\[
k=1:\;\;a_1=\pm 2,\;\;\mbox{  or  }\;\;k=4:\;\;a_i=\pm 1.
\]
The $k=1$ case would lead to $a=0$ or $a=-2$ and the classes $[V_X]=-2e_1$ or $-2(f-e_1)$.  However, $e_1$ and $f-e_1$ are representable by symplectic surfaces (Both are exceptional spheres!), hence the classes obtained for $[V_X]$ are not.

The $k=4$ case allows for 5 possible combinations of signs for the $a_i$ (assuming no ordering of the $a_i$), each leading to a new value of $a$.  As it turns out, each can be understood as the blow-up of an exceptional sphere in three points.  They are summarized in the following table: 

{\renewcommand{\arraystretch}{1.5}
\[
\begin{array}{|c|c|c|}\hline
(a_1,a_2,a_3,a_4)&a&af-\sum a_i=[V_X]\\\hline
(1,1,1,1)&1&(f-e_1)-e_2-e_3-e_4\\\hline
(-1,1,1,1)&0&e_1-e_2-e_3-e_4\\\hline
(-1,-1,1,1)&-1&e_1-(f-e_2)-e_3-e_4\\\hline
(-1,-1,-1,1)&-2&e_1-(f-e_2)-(f-e_3)-e_4\\\hline
(-1,-1,-1,-1)&-3&e_1-(f-e_2)-(f-e_3)-(f-e_4)\\\hline
\end{array}
\]}

\end{example}

\begin{example}\label{K3}
Let $\tilde X$ be a Kummer surface, i.e. a K3 surface with embedded symplectic $-2$ spheres.  Then blow up $\tilde X$ at two distinct points on a single $-2$ sphere $\tilde V$.  The new manifold $X=\tilde X\#2\;\overline{\mathbb CP^2}$ contains a symplectic submanifold $V_X$ which is a sphere of square $-4$, this is the proper transform of $\tilde V$.  Note that $V_X$ intersects each exceptional sphere in a single point.

In any elliptic K3 surface, we can obtain a $-4$-sphere by blowing up the nodal point in a fishtail fiber.  In this case, the intersection of the sphere with the exceptional fiber is 2.  
\end{example}
  
\begin{example}[Rational Manifolds containing $-4$-spheres]\label{rational}
Rational manifolds are $X=\mathbb CP^2\#k\;\overline{\mathbb CP^2}$ with $k\ge0$.  We would like to find classes $A$ such that $A$ is representable by an embedded symplectic $-4$-sphere.  We construct examples as follows:  We take an immersed pseudoholomorphic sphere, with double points as its only immersion points, in $\mathbb CP^2$.  Blowing up the nodal points and then, if necessary, blowing up further, we obtain a symplectic embedded sphere with self-intersection $-4$.

Thus we begin with a class $a[H]$ which is representable by an immersed or embedded pseudoholomorphic sphere in $\mathbb CP^2$.  Gromov-Witten theory ensures, that for each $a>0$ there exist immersed or embedded pseuodoholomorphic spherical representatives.  Denoting by $\delta$ the number of double points of a representative of $a[H]$ in $\mathbb CP^2$, we consider classes $A= a[H]-2\sum_{i=1}^\delta e_i-\sum_{i=\delta +1}^k e_i$.  Clearly, the inequality $a^2-4\delta\ge -4$ must hold, otherwise there are too many nodal points.  Using the adjunction equality for $a[H]$ to determine the number of nodal points, we obtain an upper bound for $a$:  
\[
a^2-4\delta\ge -4\;\;\Rightarrow\;\; a^2-(2a^2-6a+4)\ge -4\;\;\Rightarrow\;\; a\le 6.
\]
The following table summarizes the results:

{\renewcommand{\arraystretch}{1.5}
\[
\begin{array}{|c|c|c|c|}\hline
a&\delta&k&A\\\hline
1&0&5&[H]-\sum_{i=1}^5 e_i\\\hline
2&0&8&2[H]-\sum_{i=1}^8 e_i\\\hline
3&1&10&3[H]-2e_1-\sum_{i=2}^{10} e_i\\\hline
4&3 &11&4[H]-2\sum_{i=1}^3e_i-\sum_{i=4}^{11} e_i \\\hline
5&6 & 11&5[H]-2\sum_{i=1}^6 e_i-\sum_{i=7}^{11} e_i\\\hline
6&10 &10&6[H]-2\sum_{i=1}^{10} e_i\\\hline

\end{array}
\]}

It should be noted, that for $2\le a\le 5$, all examples above contain exceptional spheres $E$ with $E\cdot A= 0$, i.e. they are not relatively minimal.  For example, an exceptional sphere in class $H-e_1-e_2$ will not meet $A$ for $2\le g\le 4$ and an exceptional sphere in class $3H-2e_1-\sum_{i=2}^{8} e_i$ does not meet $A$ for $a=5$. 

The example with $a=1$ is particularly interesting: The first blow-up produces a fiber $H-e_1$, the second produces two exceptional spheres $e_2$ and $H-e_1-e_2$.  Blowing up the latter in further 3 points produces a $-4$-sphere.  This is the method described in Example \ref{easy} for constructing $-4$-spheres, however, in rational or ruled manifolds we have a considerable number of further exceptional spheres, in contrast to non-rational or ruled setting.

We note that the case $a=6$ is relatively minimal and that every exceptional sphere $E$ satisfies $E\cdot A=2$.  This follows from the fact, that in this case $A=-2K_X$.

The examples discussed above provide us with two methods for constructing $-4$-spheres in rational manifolds: Blowing up exceptional spheres  as in Ex. \ref{easy} and blowing up immersed pseudoholomorphic spheres.  These lead to the rational examples described in the table above, those obtained from the method described in Ex. \ref{easy} and the blow-up of a fishtail fiber in $\mathbb CP^2\#9\;\overline{\mathbb CP^2}\cong E(1)$.  

\begin{question}
Are there other classes in $\mathbb CP^2\#k\;\overline{\mathbb CP^2}$ representable by symplectic $-4$-spheres?
\end{question}

Of particular interest are the blow-up of the fishtail fiber, which is the blow up of the anticanonical class $-K$ in $\mathbb CP^2\#9\;\overline{\mathbb CP^2}$ as well as the final class in the table above, which is the blow-up of a double point in $-2K$.  This leads to the following simpler question:

\begin{question}
For which $(n,a)$ can the blow-up of a representative of $-nK_X$ in $X=\mathbb CP^2\#9\;\overline{\mathbb CP^2}$ be represented by a symplectic $-4$-sphere in class $-nK_X-ae$?
\end{question}  

These two questions (for $k\le 10$) will be of particular interest when considering the Kodaira dimension of symplectic sums along spheres, see \cite{D}.

\end{example}

\section{Minimality of Sums along Spheres\label{min}}

If $(X,V_X)$ is not relatively minimal, then clearly any symplectic sum will contain exceptional spheres.  Thus we assume in the following, that all exceptional spheres $E$ which do not meet $V_X$ have been blown down, i.e. $(X,V_X)$ is relatively minimal.  We consider minimality with regard to the list given in Section \ref{spheresum}.  We first consider sums with $S^2$-bundles.  Then we make some general statements about relative invariants in $\mathbb CP^2$ before considering the issue of minimality in sums with $\mathbb CP^2$.

\subsection{Minimality of Sums with $S^2$-bundles\label{mins2}}

This is the simplest cases and makes no use of the advanced machinery from GW-theory.

\subsubsection{$Y$ a $S^2$-bundle over a genus $g$ surface, $V_Y$ a fiber}  As $[V_X]^2=0$, $X$ must also be a $S^2$-bundle over a surface of genus $\tilde g$ and $V_X$ a fiber.  In this case the fiber sum $X\#_VY$ is again a $S^2$-bundle over a surface of genus $g+\tilde g$.  If $g+\tilde g>0$, then the sum is minimal as $X$ and $Y$ are assumed to be relatively minimal.  

If either of $X$ or $Y$ is $S^2\times S^2$, then the fiber sum does not change the diffeomorphism type.  If both $X$ and $Y$ are the nontrivial $S^2$-bundle over $S^2$, then the fiber sum contains a section of square $2$.  This determines the diffeomorphism type as that of $S^2\times S^2$, hence the sum is minimal.

It follows that $M=X\#_VY$ is minimal except possibly if $Y=S^2\times S^2$ and $V_Y$ is a fiber, in which case $M$ is minimal if and only if $X$ is minimal.

\subsubsection{$Y$ a $S^2$-bundle over a sphere, $V_Y$ a section}  In this case, as was shown in \cite{U}, $X\#_VY$ and $X$ are diffeomorphic.  Hence, the sum is minimal if and only if $X$ is minimal.

\subsection{Vanishing of Relative Invariants in $\mathbb CP^2$}

The complex projective space $\mathbb CP^2$ will play a central role in the following discussion of minimality.  We will need to calculate relative invariants with no absolute insertions, we call these purely relative.  In this section we state some vanishing results for purely relative invariants on $\mathbb CP^2$.

\begin{lemma}\label{H}[Lemma 2.7, \cite{M3}]
Consider the symplectic pair $(\mathbb CP^n, \mathbb CP^{n-1})$.  Let $\lambda\in H_2(\mathbb CP^n)$ denote the class of a line.  If $n>1$ and $d>0$ then
\[
\langle\;\vert\beta_1,...,\beta_r\rangle^{\mathbb CP^n, \mathbb CP^{n-1}}_{0,d\lambda,{\bf s}}=0
\]
for all insertions $\beta_i\in H^*(\mathbb CP^{n-1})$.

\end{lemma}

For a purely relative invariant of a symplectic 4-manifold $(X,V_X,\omega)$ to be non trivial, a necessary condition is
\[
2\left(-K_X\cdot A+g-1+r-A\cdot [V_X]\right)=\sum \mbox{ deg }\beta_i\le 2r
\]
which can be rewritten to obtain the condition
\[
(K_X\;+\;[V_X])\cdot A\;\ge\;g-1.
\]

\begin{lemma}\label{2H}
Let $(X,V_X)$ be a symplectic pair and $A\in H_2(X)$.  Assume $\dim_{\mathbb C}X=2$.  If 
\[
(K_X\;+\;[V_X])\cdot A\;<\;g-1,
\]
then all purely relative invariants $\langle\;\vert\beta_1,...,\beta_r\rangle^{X,V_X}_{g,A,{\bf s}}$ vanish.
\end{lemma}

A simple but useful corollary is

\begin{cor}\label{2happly}Let $(X,V_X)=(\mathbb CP^2, bH)$ and consider the class $A=a[H]$ for $a>0$.  
\begin{enumerate}
\item If $b=1$, all purely relative invariants in class $A$ with $g=0$ vanish.
\item If $b=2$, all purely relative invariants in class $A$ with $g=0$ vanish for $a\ge 2$.
\item If $b\le 2$, then all purely relative invariants of genus $g>0$ vanish.
\end{enumerate}

\end{cor}

\begin{proof}
The first statement is an application of Lemma \ref{H}.  The second follows from Lemma \ref{2H}:
\[
(K_X\;+\;[V_X])\cdot A\;<\;g-1\;\;\Rightarrow\;\;-[H]\cdot a[H]=-a< -1.
\]
If we consider $g>0$, then for $b\le 2$ we must have
\[
(-3+b)a<g-1,
\]
which is surely satisfied, proving the third statement.
\end{proof}


\subsection{Symplectic Sums for Embedded Curves}

In our setup, we wish to determine the GW-invariant associated to an embedded curve.  The fact that the curve in the symplectic sum is embedded will allow us to eliminate certain configurations from consideration.  

Note first, that the symplectic sum occurs in an arbitrarily small neighborhood of the hypersurface $V$.  Moreover, the sets of almost complex structures making $V_X$ resp. $V_Y$ pseudoholomorphic are only constrained on the hypersurfaces.  At any point away from $V$, the almost complex structures can be perturbed locally without constraint.  

Consider therefore a pair of classes $(A_X,A_Y)$, contributing to the sum formula, such that for generic $(J_X,J_Y)$ the moduli spaces associated to one (or both) of the classes contains a curve $(T,\phi)$ such that $\phi(T)$ is not embedded at a point away from $V$ but which is counted in the sum formula for some choice of admissible insertions.  Then there exist almost complex structures $J$ on $M=X\#_VY$, contained in the generic set obtained for the class $A$ ($\leftrightarrow (A_x,A_Y)$), for which  the absolute invariant counting embeddings $(\Sigma,\psi)$ of class $A$ in $M$ would have a contribution from a non-embedded curve $\psi(\Sigma)$.  

Thus we need only consider pairs $(A_X,A_Y)$ in the sum, such that any non-embedded behavior occurs on the hypersurface $V$.  This rules out multiple covers as contributors in the sum formula.

%

\subsection{Minimality of Sums with $\mathbb CP^2$}

We now show, that under blow-down we obtain no new exceptional spheres while for the rational blow-down it is possible to obtain an exceptional sphere via gluing.

Consider a pseudoholomorphic map $(T,\phi)$ into $\mathbb CP^2$.  Let $C_1$ and $C_2$ correspond to the image under $\phi$ of two disjoint connected components of $T$.  Then $C_1$ and $C_2$ intersect in $\mathbb CP^2$, so that even if $\phi$ restricted to each component is an embedding, the image $\phi(T_1\cup T_2)$ is not an embedded curve.  If such a curve is to contribute to the absolute invariant for an embedded curve in the sum formula, then we shall need to ensure that $C_1\cap C_2\subset V_{\mathbb CP^2}$.  However, in the GW-invariants, we have no tools available to ensure this, as homological insertions will not suffice.  We may thus assume, that in the following arguments involving an embedded connected exceptional sphere, the maps on the $\mathbb CP^2$ side are from connected Riemann surfaces $\Sigma$ and $\phi$ is an embedding.

\subsubsection{$(Y,V_Y)=(\mathbb CP^2,H)$}   In this case $V_X\subset X$ represents an exceptional sphere. 

Assume first that $X$ is not rational or ruled.  Then no two exceptional spheres intersect, see Lemma \ref{inters}, hence, having assumed $(X,V_X)$ is relatively minimal, in this case $X= X_m\#\overline{\mathbb CP^2}$ for some minimal manifold $ X_m$.  The fiber sum $M=X\#_VY$ is just the blow down of $X$, i.e. $M=X_m$.  Hence the result is always minimal for any relatively minimal $(X,V_X)$.  

If $X$ is rational or ruled, we can no longer assume that $V_X$ is the only exceptional sphere in $X$.  However, assuming relative minimality implies that all other exceptional spheres must satisfy $E\cdot [V_X]\ne 0$, more precisely Lemma 3.5, \cite{LL2} shows that $E\cdot [V_X]> 0$.

We will use the sum formula to show that $M$ is a minimal manifold.  Assume that $M$ is not minimal and let $E$ denote an exceptional class in $M$.  The sum formula for this configuration is as follows:
\[
1=\langle\;\rangle^{M}_{0,E}=\sum \langle\;\vert\beta_1,...,\beta_r\rangle^{X,V_X}_{T_X,{\bf s}}\langle\;\vert \hat\beta_1,...,\hat\beta_r\rangle^{\mathbb CP^2,H}_{T_{\mathbb CP^2},{\bf s}}
\]
where we have no absolute insertions, see Ex. \ref{-1}.  

As discussed above, we may assume $T_{\mathbb CP^2}$ has one connected component.  Consider the invariant $ \langle\;\vert \hat\beta_1,...,\hat\beta_r\rangle^{\mathbb CP^2,H}_{0,B,{\bf s}}$ on $\mathbb CP^2$.  All classes $B=a[H]$ must satisfy $a>0$ to be effective.  For classes with $a>0$, Lemma \ref{H} shows that the purely relative invariant vanishes.

Thus we see that no configuration exists such that the sum formula provides a non-zero answer, hence $M$ must be minimal.

\subsubsection{$(Y,V_Y)=(\mathbb CP^2,2H)$}

The rational blow-down of a symplectic $-4$-sphere can produce a non-minimal manifold $M$.  This section provides three examples of this behavior.  Assuming $(X,V_X)$ is relatively minimal, the appearance of an exceptional sphere in the fiber sum is rather more subtle than in the cases of $S^2$-bundles, i.e. it is not just dependent on the minimality or non-minimality of $X$. 

Cor. 1.7, \cite{G} states, that embedded symplectic surfaces in $X$ and $Y$ intersecting $V$ transversely and positively can be glued to obtain a symplectic surface in $M$.  This is of course at the core of the GW-sum formula, but in the following examples we make direct use of this result.  

We begin with an example which makes no use of an exceptional sphere in the initial manifold $X$ to produce an exceptional sphere in the sum.

\begin{example}\label{-20}
Let $X=\mathbb CP^2\#10\;\overline{\mathbb CP^2}$ and $[V_X]=3[H]-2e_1-\sum_{i=2}^{10} e_i$.  Then the classes $2[H]-\sum_{i=1}^4e_i$ and $e_1-e_2$ are representable by symplectic embedded spheres of self-intersection $0$ and $-2$ respectively each intersecting a hypersurface $V_X$ in a single transverse positive point.  The rational blow-down of $V_X$ will produce a symplectic exceptional sphere from these two spheres when glued to a line $H$ in $\mathbb CP^2$.  However, note that $(e_1-e_2)\cdot e_1=-1$, thus there exists no $J$ making both spheres $J$-holomorphic at the same time.  Moreover, the initial configuration ensuring that such curves exist is highly non-generic:  The sphere representing $2[H]$ must intersect the immersed sphere representing $3[H]$ in the double point and in such a manner, that after blowing up this double point the manifolds representing $e_1$, $3[H]-2e_1$ and $2[H]-e_1$ all intersect in a single point.  It is this point which is then blown-up to obtain the $-2$-sphere representing $e_1-e_2$.

\begin{figure}[H]
\centering
\begin{tikzpicture}
\draw (-2,4) node{$X$};
\draw (3,4) node{$\mathbb{C}P^2$};
\draw (0,0)--(0,4) node[anchor=east]{$V_X$};
\draw (0,1)--(-2,1) node[anchor=north]{$C_2$};
\draw (0,3)--(-2,3) node[anchor=north]{$C_1$};
\draw (1,0)--(1,4) node[anchor=west]{$V_{\mathbb CP^2}$};
\draw[dotted] (0,3)--(1,3);
\draw[dotted] (0,1)--(1,1);
\draw (1,1) arc(-90:90:1cm);
\draw (2,2) node[anchor=south west]{$H$};
\end{tikzpicture}
\caption{\label{pic1}This configuration of curves leads to non-minimal $M=X\#_VY$ in Ex. \ref{-20} and \ref{nonminimal}.  }
\end{figure}
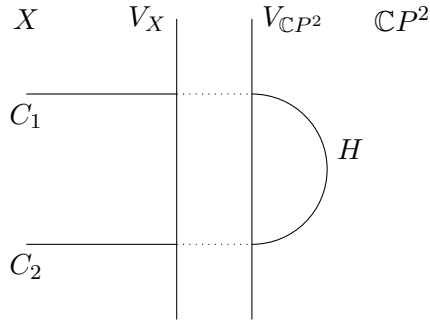

\end{example}

Note that the configuration does have at least two exceptional spheres meeting $V_X$ as described in Thm. \ref{minimal}.  In the next section it will be seen that such a configuration always exists in $X$ relative to $V_X$ if the manifold $M$ is non-minimal (and $(X,V_X)$ is relatively minimal), i.e. the configuration described in Thm. \ref{minimal} is a criterion for the non-minimality of the sum. 

We now consider examples with exceptional spheres contributing to the construction in the sum.  We shall see two different configurations which lead to non-minimal sums.  The first is the configuration mentioned in Thm. \ref{minimal}.

\begin{example}\label{nonminimal}
The only classes with embedded spheres as representatives in $\mathbb CP^2$ are $A_Y\in\{[H],2[H]\}$.  If $A_Y=[H]$, then, applying Lemma \ref{formulas}, we obtain that $A_X^2=-2$ and $A_X\cdot[V_X]=2$.  Letting the class $A_X$ be represented by two disjoint spheres $A_1$ and $A_2$ such that $(A_1+A_2)^2=A_1^2+A_2^2=-2$ we obtain $K_X\cdot A_X=-4-(A_1^2+A_2^2)=-2$, and thus $K_{X\#_VY}\cdot A=-1$.  

Such a configuration of curves can be found in Ex. \ref{K3}.  In the manifold $X$, each exceptional sphere intersects $V_X$ positively and transversally in a single point of order 1.  Hence, the fiber sum $M=X\#_{V_X=2H}\mathbb CP^2$ contains a single exceptional sphere generated by the two exceptional spheres and $H\subset \mathbb CP^2$, see Fig. \ref{pic1}.    
\end{example}

\begin{example}
Consider a configuration of exceptional spheres in $V$ such that $e_1\cdot[V_X]=1$ and $e_2\cdot [V]=2$.  (Such a configuration is hidden in Ex. \ref{-20}.)  Then the rational blow-down of $V_X$ will be non-minimal.  This can be seen as follows:  Lemma 5.1, \cite{G} allows us to blow down $e_1$ in $X$ while blowing up a point in $2H\subset \mathbb CP^2$ without changing the diffeomorphism type of the fiber sum $M$, i.e. $M=X\#_{V_X}\mathbb CP^2$ is diffeomorphic to $(X\#_{e_1}\mathbb CP^2)\#_{\tilde V}(\mathbb CP^2\#\overline{\mathbb CP^2})$ where $\tilde V$ is the blow-up of $2H$ in a point.  The latter is the sum of a non-minimal manifold $X\#_{e_1}\mathbb CP^2$ with an $S^2$-bundle along a section.  Thus, by the results in Section \ref{mins2}, $M$ is non-minimal.  The precise appearance of an exceptional sphere can be seen in the following diagrams:

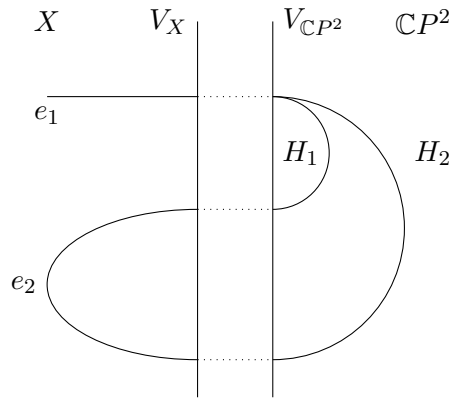
\begin{figure}[H]
\centering
\begin{tikzpicture}
\draw (-2,5) node{$X$};
\draw (3,5) node{$\mathbb{C}P^2$};
\draw (0,0)--(0,5) node[anchor=east]{$V_X$};

\draw (0,2.5) arc(90:270:2cm and 1cm);
\draw (-2,1.5) node[anchor=east]{$e_2$};
\draw (0,4)--(-2,4) node[anchor=north]{$e_1$};

\draw (1,0)--(1,5) node[anchor=west]{$V_{\mathbb CP^2}$};

\draw[dotted] (0,2.5)--(1,2.5);
\draw[dotted] (0,0.5)--(1,0.5);
\draw[dotted] (0,4)--(1,4);

\draw (1,2.5) arc(-90:90:0.75cm);

\draw (1.75,3.25) node[anchor=east]{$H_1$};
\draw (1,0.5) arc(-90:90:1.75cm);

\draw (2.75,3.25) node[anchor=west]{$H_2$};
\end{tikzpicture}
\caption{ The rational blow-down $M=X\#_{V_X}\mathbb CP^2$ of the $-4$-sphere $V_X$. }
\end{figure}

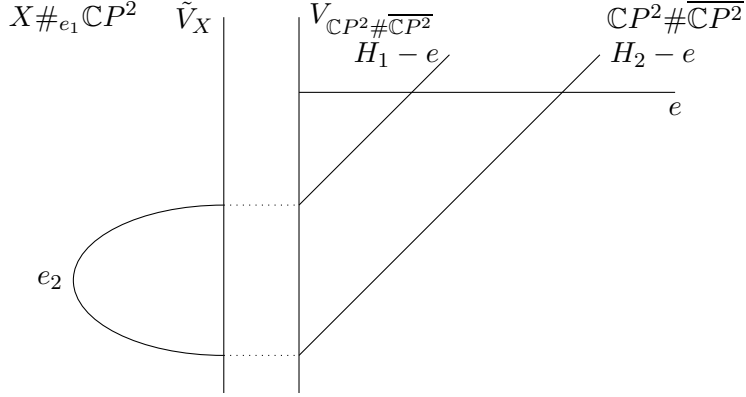
\begin{figure}[H]
\centering
\begin{tikzpicture}
\draw (-2,5) node{$X\#_{e_1}\mathbb CP^2$};
\draw (6,5) node{$\mathbb{C}P^2\#\overline{\mathbb CP^2}$};
\draw (0,0)--(0,5) node[anchor=east]{$\tilde V_X$};

\draw (0,2.5) arc(90:270:2cm and 1cm);
\draw (-2,1.5) node[anchor=east]{$e_2$};
\draw (1,4)--(6,4) node[anchor=north]{$e$};

\draw (1,0)--(1,5) node[anchor=west]{$\tilde V_{\mathbb CP^2\#\overline{\mathbb CP^2}}$};

\draw[dotted] (0,2.5)--(1,2.5);
\draw[dotted] (0,0.5)--(1,0.5);


\draw (1,2.5)-- (3,4.5) node[anchor=east]{$H_1-e$};

\draw (1,0.5)--(5,4.5) node[anchor=west]{$H_2-e$};
\end{tikzpicture}
\caption{  $(X\#_{e_1}\mathbb CP^2)\#_{\tilde V}(\mathbb CP^2\#\overline{\mathbb CP^2})$}
\end{figure}

\end{example}

\subsection{Rational Blow-down}  The main purpose of this section is to provide a criterion which will allow us to determine whether the rational blow-down will be minimal or not.  This again makes use of the sum formula for Gromov-Witten invariants developed in \cite{LR}.  We show that the configuration in Ex. \ref{nonminimal} is the only configuration leading to an exceptional sphere in the sum for a generic choice of almost complex structure $J$ among those making the $-4$-sphere $J$-holomorphic.  

The sum formula is as follows:
\[
1=\langle\;\rangle^{M}_{0,E}=\sum \langle\;\vert\beta_1,...,\beta_r\rangle^{X,V_X}_{T_X,{\bf s}}\langle\;\vert \hat\beta_1,...,\hat\beta_r\rangle^{\mathbb CP^2,2H}_{T_{\mathbb CP^2},{\bf s}}
\]
where $\hat\beta_i$ is dual to $\beta_i$ in $H^*(V)$.  We may again restrict ourselves to only connected embedded curves in the $\mathbb CP^2$ invariant as in the case $(Y,V_Y)=(\mathbb CP^2,H)$.

 Cor \ref{2happly} shows, that only if $B=[H]$ does the invariant $ \langle\;\vert \hat\beta_1,...,\hat\beta_r\rangle^{\mathbb CP^2,2H}_{0,B,{\bf s}}$ possibly contribute non-trivially.  Moreover, in this case we actually have
\[
\sum_{i=1}^r\mbox{ deg }\beta_i\;=\;2\left(-(-3[H])\cdot [H]-1+r-2[H]\cdot [H]\right)\;=\;2r.
\]
Thus the relative insertions in the $\mathbb CP^2$ invariant for the class $B=[H]$ must be one of the following:
\begin{enumerate}
\item[(a)] a single point insertion (deg $\beta_1=2$) with contact order 2 or
\item[(b)] two point insertions (deg $\beta_i=2$), at each point with order 1.
\end{enumerate}
Moreover, the restrictions on the $\mathbb CP^2$ side lead to the following conditions on the class $A_X$ (see Lemma \ref{formulas}):
\begin{equation}\label{xdata}
A_X\cdot [V_X]=2,\;\;\;A_X^2=-2\;\;\mbox{   and   }\;\;K_X\cdot A_X=-2.
\end{equation}

{\bf One contact point of order $\bf(2)$:}  Consider any embedded curve $C_X=\phi(T_X)$ representing $A_X$.  The curve $C_X$ must be connected, as we have only one contact 
 point with $V_X$ and the exceptional sphere in $M$ is connected.  Moreover, $C_X$ cannot be the image of a simple $J$-holomorphic map, as otherwise the adjunction formula would imply
 \[
 K_X\cdot A_X\ge 0
 \]
 contradicting Eq. \ref{xdata}.

Thus no pseudoholomorphic maps meeting $V_X$ in a single point of contact order $(2)$ contribute in the sum formula.

{\bf Two distinct contact points:}  Denote by $A_*=[C_*]\in H_2(X)$ the classes of the two disjoint curves $C_1$ and $C_2$.  These classes must satisfy
\[
A_1^2+A_2^2=-2\mbox{ and }A_1\cdot[V_X]=1=A_2\cdot[V_X]
\] 
while $A_1\cdot A_2=0$.  Both  curves must have genus 0.   

The two curves are chosen disjoint and we need to calculate the invariant $\langle\;\vert\beta_1,...,\beta_r\rangle^{X,V_X}_{T_X,(1,1)}$ where the graph $T_X$ consists of two distinct nodes each with a tail and no edges.  Due to the definition of this invariant, we consider the existence of each curve separately.

The dimension of the strata of curves representing $A_1$ is
\[
\dim_{\mathbb C}=A_1^2+1.
\] 
For this to have non-negative dimension $A_1^2\ge -1$. The same dimension statement holds for $A_2$ as well.  Hence we obtain $A_1^2=A_2^2=-1$.  This corresponds to two disjoint exceptional spheres $C_1$ and $C_2$.

This exhausts all possible configurations and allows us to state the following cirterion:

\begin{lemma}Consider the rational blow-down $M$ of $(X,V_X)$.  $M$ is not minimal if:
\begin{itemize}
\item $(X,V_X)$ is not relatively minimal or 
\item $X$ contains  2 disjoint distinct exceptional spheres $E_i$ each meeting $V_X$ transversely and positively in a single point with $[E_i]\cdot [V_X]=1$.
\end{itemize}

\end{lemma}

\begin{remark} This result is no longer true in the smooth category.  As was shown in Ex. 2 of Section 3, \cite{FS}, a K3-surface contains a smooth embedded $-4$-sphere which, when it is blown down, produces a manifold diffeomorphic to $3\mathbb CP^2\#18\;\overline{\mathbb CP^2}$.

\end{remark}


\begin{thebibliography}{alpha}

%
\bibitem{D}Dorfmeister, Josef G. {\it Kodaira Dimension of Symplectic Fiber Sums along Spheres.} in preparation.

\bibitem{DL2}Dorfmeister, Josef G.; Li, Tian-Jun.  {\it Relative Ruan and Gromov-Taubes Invariants of Symplectic 4-Manifolds.}  arXiv:0912.0651.



\bibitem{FS} Fintushel, Ronald;  Stern, Ronald J.  {\it Rational blowdowns of smooth $4$-manifolds.}
 J. Differential Geom.  46  (1997),  no. 2, 181--235.

%

\bibitem{G} Gompf, Robert E.  {\it A new construction of symplectic manifolds.}
 Ann. of Math. (2)  142  (1995),  no. 3, 527--595.


%
%
\bibitem{HLR}Hu, Jianxun ;  Li, Tian-Jun ;  Ruan, Yongbin . {\it Birational cobordism invariance of uniruled symplectic manifolds.}
 Invent. Math.  172  (2008),  no. 2, 231--275.


\bibitem{IP4}  Ionel, Eleny-Nicoleta ;  Parker, Thomas H.  {\it Relative Gromov-Witten invariants.}
 Ann. of Math. (2)  157  (2003),  no. 1, 45--96.


\bibitem{IP5}  Ionel, Eleny-Nicoleta ;  Parker, Thomas H.  {\it The symplectic sum formula for Gromov-Witten invariants.}
 Ann. of Math. (2)  159  (2004),  no. 3, 935--1025.




\bibitem{Le}Lerman, Eugene. {\it Symplectic cuts.}
 Math. Res. Lett.  2  (1995),  no. 3, 247--258.


\bibitem{LR}Li, An-Min;  Ruan, Yongbin. {\it Symplectic surgery and Gromov-Witten invariants of Calabi-Yau
 3-folds.}
 Invent. Math.  145  (2001),  no. 1, 151--218.

\bibitem{Li1}Li, Jun . {\it Lecture notes on relative GW-invariants.}
 Intersection theory and moduli, 
 41--96 (electronic), ICTP Lect. Notes, XIX, Abdus Salam Int. Cent. Theoret. Phys., Trieste,  2004.


\bibitem{TJL2}Li, Tian-Jun. {\it Smoothly embedded spheres in symplectic {$4$}-manifolds.} Proc. Amer. Math. Soc. 2(1999), 609-613.


\bibitem{L1}Li, Tian-Jun. {\it Symplectic 4-manifolds with Kodaira dimension zero.}
 J. Differential Geom.  74  (2006),  no. 2, 321--352.






\bibitem{Liun}Li, Tian-Jun.  unpublished notes

\bibitem{LL}Li, Tian-Jun;  Liu, Ai-Ko.  {\it General wall crossing formula.}
 Math. Res. Lett.  2  (1995),  no. 6, 797--810.


\bibitem{LL2}Li, Tian-Jun;  Liu, Ai-Ko. {\it Uniqueness of symplectic canonical class, surface cone and symplectic
 cone of 4-manifolds with $b\sp +=1$.}
 J. Differential Geom.  58  (2001),  no. 2, 331--370.

\bibitem{LS}Li, Tian-Jun;  Stipsicz, Andr\'as I.  {\it Minimality of certain normal connected sums.}
 Turkish J. Math.  26  (2002),  no. 1, 75--80.

\bibitem{LU}Li, Tian-Jun; Usher, Michael. {\it Symplectic forms and surfaces of negative square.}
J. Symplectic Geom. 4 (2006), no. 1, 71--91. 


\bibitem{liu}Liu, Ai-Ko. {\it Some new applications of general wall crossing formula, Gompf's
 conjecture and its applications.}
 Math. Res. Lett.  3  (1996),  no. 5, 569--585.




\bibitem {MW}McCarthy, John D.;  Wolfson, Jon G.  {\it Symplectic normal connect sum.}
 Topology  33  (1994),  no. 4, 729--764.


\bibitem{M} McDuff, Dusa. {\it The structure of rational and ruled symplectic $4$-manifolds.}
 J. Amer. Math. Soc.  3  (1990),  no. 3, 679--712.

\bibitem{M2}McDuff, Dusa. {\it
Immersed spheres in symplectic $4$-manifolds.}
Ann. Inst. Fourier (Grenoble) 42 (1992), no. 1-2, 369--392.

 
\bibitem{M3}McDuff, Dusa. {\it Comparing Absolute and relative Gromov-Witten invariants.} arXiv.org:0809.3534.


\bibitem{MS} McDuff, Dusa;  Salamon, Dietmar . \textit{Introduction to symplectic topology.}
Second edition.
Oxford Mathematical Monographs. The Clarendon Press, Oxford University Press, New York,  1998. x+486 pp.


\bibitem{S}Stipsicz, Andr\'as I. {\it Indecomposability of certain Lefschetz fibrations.}
 Proc. Amer. Math. Soc.  129  (2001),  no. 5, 1499--1502.


\bibitem{T1}Taubes, Clifford Henry. {\it $\rm SW\Rightarrow Gr$: from the Seiberg-Witten equations to
 pseudo-holomorphic curves.}
 Seiberg Witten and Gromov invariants for symplectic 4-manifolds,
 1--97, First Int. Press Lect. Ser., 2, Int. Press, Somerville, MA,  2000.

\bibitem{T}Taubes, Clifford Henry. {\it Counting pseudo-holomorphic submanifolds in dimension 4.}
 Seiberg Witten and Gromov invariants for symplectic 4-manifolds,
 99--161, First Int. Press Lect. Ser., 2, Int. Press, Somerville, MA,  2000.



\bibitem{U} Usher, Michael. {\it Minimality and symplectic sums.}
 Int. Math. Res. Not.  2006, Art. ID 49857, 17 pp.




\bibitem{W}Witten, Edward . \textit{Monopoles and four-manifolds.}
 Math. Res. Lett.  1  (1994),  no. 6, 769--796.


\end{thebibliography}
\end{document}